\documentclass[12pt,twoside]{amsart} 
 \title{Remarks on  the non-vanishing conjecture}
\author{Yoshinori  Gongyo}
\address{Graduate School of Mathematical Sciences, 
the University of Tokyo, 3-8-1 Komaba, Meguro-ku, Tokyo 153-8914, Japan.}
\email{gongyo@ms.u-tokyo.ac.jp}
\date{\today, version 2.04}
\thanks{The author is partially supported by Grant-in-Aid for JSPS Fellows $\sharp$22$\cdot$7399.}

\newcommand{\Supp}[0]{{\operatorname{Supp}}}

\usepackage{pxfonts}

\usepackage{latexsym}
\usepackage{amsmath}
\usepackage{amssymb}
\usepackage{amsthm}
\usepackage{amscd}
\usepackage{enumerate} 
\usepackage{amssymb} 
\usepackage{mathrsfs}          
\usepackage[all]{xy}
\usepackage{color}

\newtheorem{thm}{Theorem}[section]

\newtheorem{lem}[thm]{Lemma}

\newtheorem{rem}[thm]{Remark}

\newtheorem{conj}[thm]{Conjecture}

\newtheorem{note}[thm]{Notation}

\newtheorem{cl}[thm]{Claim}
 
\theoremstyle{definition}

\newtheorem{defi}[thm]{Definition}

\newtheorem{case}{Case}

\newtheorem*{ack}{Acknowledgments}

\begin{document}
\bibliographystyle{amsalpha+}
 
 \maketitle
 
  \begin{abstract}We discuss a difference between the rational and the real non-vanishing conjecture for pseudo-effective log canonical divisors of log canonical pairs. We also show the log non-vanishing theorem for rationally connected varieties under assuming Shokurov's ACC conjectures.
 \end{abstract}

\section{Introduction}\label{real-intro}
Throughout this article, we work over $\mathbb{C}$, the complex number field. We will freely use the standard notations in \cite{kamama}, \cite{komo}, and \cite{bchm}. In this article we deal a topic related to the abundance conjecture: 
 
\begin{conj}[Abundance conjecture]\label{conj-abun-nef}Let $(X,\Delta)$ be a projective log canonical pair such that $\Delta$ is an effective $\mathbb{Q}$-divisor and $K_X+\Delta$ is nef. Then $K_X+\Delta$ is semi-ample.  
 
  \end{conj}
  
Let $\mathbb{K}$ be the real number field $\mathbb{R}$ or the rational number field $\mathbb{Q}$. The following conjecture seems to be the most difficult and important conjecture for proving Conjecture \ref{conj-abun-nef}:

 \begin{conj}[Non-vanishing conjecture]\label{conj-non-vani-k}Let $(X,\Delta)$ be a projective log canonical pair such that $\Delta$ is an effective  $\mathbb{K}$-divisor and $K_X+\Delta$ is pseudo-effective. Then there exists an effective $\mathbb{K}$-divisor $D$ such that $D \sim_{\mathbb{K}}K_X+\Delta$.  
 
  \end{conj}
  Note that the above conjecture is obviously true for big log canonical divisors. Thus it is important for pseudo-effective log canonical divisors which are not big.
In this article, we study a difference between Conjecture \ref{conj-non-vani-k} for $\mathbb{K}=\mathbb{Q}$ and $\mathbb{R}$. One of the importence of Conjecture \ref{conj-non-vani-k} for $\mathbb{K} =\mathbb{R}$ is motivated in Birkar's framework on the existence of minimal models \cite{bir-exiII}.  In his construction, Conjecture \ref{conj-non-vani-k} must be formulated for {\em log canonical} pairs with {\em $\mathbb{R}$-boundary} when we construct minimal models for even smooth projective varieties. For reducing Conjecture \ref{conj-non-vani-k} in the case where $\mathbb{K}=\mathbb{R}$ to the case where $(X,\Delta)$ is kawamata log terminal with $\mathbb{Q}$-boundary, we need the following two conjectures (cf.~Lemma \ref{term pe fibration lemma}):

\begin{conj}[Global ACC conjecture, cf.{~\cite[Conjecture 2.7]{bSh}, \cite[Conjecture 8.2]{dhp-ext}}]\label{gacc} Let  $d\in \mathbb{N}$ and  $I \subset [0,1]$ a set satisfying the DCC. Then there is a finite subset $I_0 \subset I$ such that if 
 \begin{enumerate}
\item $X$ is a projective variety of dimension $d$,
\item $(X,\Delta  )$ is log canonical,
\item $\Delta =\sum \delta _i \Delta _i $ where $\delta _i \in I$,
\item $K_X+\Delta \equiv 0$,
 \end{enumerate}
then $\delta _i \in I_0$. 
 \end{conj}
 
  \begin{conj}[ACC conjecture for log canonical thresholds, cf. {\cite[Conjecture 1.7]{bSh}, \cite[Conjecture 8.4]{dhp-ext}}]\label{acclct} Let $d \in \mathbb{N}$, $\Gamma \subset [0,1]$ be a set satisfying the DCC and, let $S \subset \mathbb{R} _{\geq 0}$ be a finite set. Then the set
  $$\{{\rm lct}(X,\Delta;D )|\ (X,\Delta )\ {\rm is\ lc},\ \dim X=d,\ \Delta \in \Gamma,\ D\in S\}$$
  satisfies the ACC. Here $D$ is $\mathbb{R}$-Cartier and $\Delta\in \Gamma$ (resp.~$D\in S$) means $\Delta=\sum \delta _i\Delta _i$ where $\delta _i\in \Gamma$ (resp.~$D=\sum d_iD_i$ where $d_i\in S$)  and ${\rm lct}(X,\Delta;D )={\rm sup}\{ t\geq 0|(X,\Delta +tD)\ {\rm is\ lc}\}$. 
  \end{conj}

The proofs of the above two conjectures are announced by Hacon--$\mathrm{M^{c}}$Kernan--Xu.  See \cite[Remark 8.3]{dhp-ext}.

 Namely the main theorem of this article is the following:

\begin{thm}\label{main-real}Assume that the global ACC conjecture (\ref{gacc}) in dimension $\leq n$, the ACC conjecture for log canonical thresholds (\ref{acclct}) in dimension $\leq n$, and the abundance conjecture (\ref{conj-abun-nef}) in dimension $\leq n-1$. Then the non-vanishing conjecture (\ref{conj-non-vani-k}) for $n$-dimensional klt pairs in the case where $\mathbb{K}=\mathbb{Q}$ implies that for $n$-dimensional lc pairs in the case where $\mathbb{K}=\mathbb{R}$.   

\end{thm}

Under assuming that Conjecture \ref{gacc} and Conjecture \ref{acclct} hold, by combining with \cite[Theorem 8.8]{dhp-ext}, we can reduce Conjecture \ref{conj-non-vani-k} in the case where $\mathbb{K}=\mathbb{R}$ to the case where $X$ is smooth and $\Delta=0$. The proof of Theorem \ref{main-real} was inspired by Section $8$ in \cite{dhp-ext} and discussions the author had with Birkar in Paris.

In Section \ref{RC-non}, we also show the log non-vanishing theorem (= Theorem \ref{main-RC-non-van}) for rationally connected varieties by the same argument.

\begin{ack}
The author wishes to express his deep gratitude to Professor Caucher Birkar for discussions. He thanks Professor Osamu Fujino for careful reading the first version of this article and pointing out several mistakes. He would like to thank Professor Christpher D. Hacon for answering his question about Section $8$ in \cite{dhp-ext}, and Professors Hiromichi Takagi and Chenyang Xu  for various comments.  He also thanks Professor Claire Voisin and Institut de Math\'ematiques de Jussieu (IMJ) for their hospitality. He partially worked on this article when he stayed at IMJ. He is grateful to it for its hospitality. 
\end{ack}

\begin{note}\label{notation} A variety $X/Z$ means that a quasi-projective normal variety $X$ is projective over a quasi-projective variety $Z$. A rational map $f: X \dashrightarrow Y/Z $ denotes a rational map $X \dashrightarrow Y$ over  $Z$.  For a contracting birational map $X  \dashrightarrow  Y/Z$ and   an $\mathbb{R}$-Weil divisor $D$ on $X$, an $\mathbb{R}$-Weil divisor $D^Y$ means the strict transform of $D$ on $Y$.

\end{note}

\section{On the existence of minimal models after Birkar}\label{real-exi-mm}In this section we introduce the definitions of minimal models in the sense of Birkar--Shokurov and some results on the existence of minimal models after Birkar.

\begin{defi}[cf. {\cite[Definition 2.1]{bir-exiII}}]\label{real-def-minimalmodelsenseofBS}
A pair $(Y/Z,B_Y)$ is a \emph{log birational model} of $(X/Z,B)$ if we are given a birational map
$\phi\colon X\dashrightarrow Y/Z$ and $B_Y=B^\sim+E$ where $B^\sim$ is the birational transform of $B$ and 
$E$ is the reduced exceptional divisor of $\phi^{-1}$, that is, $E=\sum E_j$ where $E_j$ are the
exceptional over $X$ prime divisors on $Y$. A log birational model $(Y/Z,B_Y)$ is a \emph{nef model} of $(X/Z,B)$ if in addition\\\
(1)$(Y/Z,B_Y)$ is $\mathbb{Q}$-factorial dlt, and\\\
(2)$K_Y+B_Y$ is nef over $Z$.\\\
And  we call a nef model $(Y/Z,B_Y)$ a \emph{log minimal model of $(X/Z,B)$ in the sense of Birkar--Shokurov} if in addition\\\
(3) for any prime divisor $D$ on $X$ which is exceptional over $Y$, we have
$$
a(D,X,B)<a(D,Y,B_Y).
$$
\end{defi}

\begin{rem}\label{rem_real_bir1}The followings are remarks:
\begin{itemize}
\item[(1)] Conjecture \ref{conj-non-vani-k} in the case where the dimension $\leq n-1$ and $\mathbb{K}=\mathbb{R}$  implies the existence of relative log minimal models in the sense of Birkar--Shokurov over a quasi-projective base $Z$ for effective dlt pairs over $Z$ in dimension $n$. See \cite[Corollary 1.7 and Theorem 1.4]{bir-exiII}.

\item[(2)] Conjecture \ref{conj-non-vani-k} in the case where the dimension $\leq n-1$ and $\mathbb{K}=\mathbb{R}$ implies Conjecture \ref{conj-non-vani-k} in the case where the dimension $\leq n$ and $\mathbb{K}=\mathbb{R}$ over a non-point quasi-projective base $Z$. See \cite[Lemma 3.2.1]{bchm}.

\item[(3)] When $(X/Z,B)$ is purely log terminal, a log minimal model of $(X/Z,B)$ in the sense of Birkar--Shokurov is the traditional one as in \cite{komo} and \cite{bchm}. See \cite[Remark 2.6]{bir-exiI}.

\end{itemize}

\end{rem}

\section{Proof of Theorem \ref{main-real}}\label{section-proof of main real}
In this section, we give the proof of Theorem \ref{main-real}. The proof of following lemma is essentially same as the proof of  \cite[Proposition 8.7]{dhp-ext}.
\begin{lem}[cf.{ \cite[Proposition 8.7]{dhp-ext}}]\label{term pe fibration lemma}Assume that the global ACC conjecture (\ref{gacc}) in dimension $\leq n$ and the ACC conjecture for log canonical thresholds (\ref{acclct}) in dimension $\leq n$. Let $(X,\Delta)$ be a $\mathbb{Q}$-factorial projective dlt pair such that $\Delta$ is an $\mathbb{R}$-divisor and $K_X+\Delta$ is pseudo-effective. Suppose that there exists a sequence of effective divisors $\{\Delta_i\}$ such that $\Delta_i\leq \Delta_{i+1}$, $K_X+\Delta_i$ is not pseudo-effective for any $i\geq 0$, and
$$\lim_{i \to \infty}\Delta_i=\Delta.$$
Then there exists a contracting birational map $\varphi:X \dashrightarrow X'$ such that there exists a projective morphism $f':X' \to Z$ with connected fibers satisfying:
\begin{itemize}
\item[(1)] $(X',\Delta')$ is $\mathbb{Q}$-factorial log canonical and $\rho(X'/Z)=1$,

\item[(2)] $K_{X'}+\Delta' \equiv_{f'} 0$, 

\item[(3)] $\Delta'-\Delta'_i$ is $f'$-ample for some $i$, and

\item[(4)] $\mathrm{dim}\,X > \mathrm{dim}\,Z,$

\end{itemize}
where $\Delta'$ and $\Delta'_i$ are the strict transform of $\Delta$ and $\Delta_i$ on $X'$.
\end{lem}
\begin{proof}Set $\Gamma_i=\Delta-\Delta_i$. Then $K_X+\Delta_i+x\Gamma_i$ is also not pseudo-effective for every non-negative number $x<1$. For any $i$ and non-negative number $x<1$, we can take a Mori fiber space $f_{x,i}:Y_{x,i} \to Z_{x,i}$ of $(X,\Delta_i+x\Gamma_i)$ by \cite{bchm}. Then there exists a positive number $\eta_{x,i}$ such that 
$$K_{Y_{x,i}}+ \Delta^{Y_{x,i}}_{i}+ \eta_{x,i} \Gamma^{Y_{x,i}}_{i} \equiv_{f_{x,i}} 0.$$
Note that $x <\eta_{x,i} \leq 1$ and $x\leq \mathrm{lct}(Y_{x,i}, \Delta^{Y_{x,i}}_{i};\Gamma^{Y_{x,i}}_{i})$ since $K_{Y_{x,i}}+ \Delta^{Y_{x,i}}_{i}+ \Gamma^{Y_{x,i}}_{i}$ is pseudo-effective.

\begin{cl}\label{rem-real-lct-term}
 When we consider an increasing sequence $\{x_j\}$ such that 
$$\lim_{j \to \infty}x_j=1,$$
it holds that $\mathrm{lct}(Y_{x_j,i}, \Delta^{Y_{x_j,i}}_{i};\Gamma^{Y_{x_j,i}}_{i})\geq 1$ for $j\gg 0$ 
\end{cl}

\begin{proof}[Proof of Claim \ref{rem-real-lct-term}] Put 
$$l_{j,i}=\mathrm{lct}(Y_{x_j,i}, \Delta^{Y_{x_j,i}}_{i};\Gamma^{Y_{x_j,i}}_{i}).$$
Assume by contradiction that $l_{j,i}< 1$
for some infinitely many $j$. Fix such an index $j_0$. Then we take a $j_1$ such that $l_{j_1,i}< 1$ and $l_{j_0,i}< x_{j_1} < 1.$ Since $l_{j_1,i}< 1$, we take $l_{j_1,i}< x_{j_2} < 1.$ By repeating, we construct  increasing sequences $\{x_{j_k}\}_{k}$ and $\{l_{j_k,i}\}_{k}$. Actually this is a contradiction to Conjecture \ref{acclct}.
\end{proof}
Thus, for any $i$, there exists non-negative number $y_i<1$ such that 
$$\lim_{i \to \infty}y_i=1,$$
$$K_{Y_{y_i,i}}+ \Delta^{Y_{y_i,i}}_{i}+ \eta_{y_i,i} \Gamma^{Y_{y_i,i}}_{i} \equiv_{f_{y_i,i}} 0,$$
and $(Y_{y_i,i}, \Delta^{Y_{y_i,i}}_{i}+ \eta_{y_i,i} \Gamma^{Y_{y_i,i}}_{i})$ is log canonical from Claim \ref{rem-real-lct-term}. Set $\Omega_{i}=\Delta_{i}+ \eta_{y_i,i} \Gamma_{i}$ and $Y_i=Y_{y_i,i}$. Then we see the following:

\begin{cl}\label{rem-real-gacc-term}
 It holds that $K_{Y_i}+\Delta^{Y_i}$ is $f_i:=f_{y_{i}, i}$-numerical trivial for some $i$.
\end{cl}
\begin{proof}[Proof of Claim \ref{rem-real-gacc-term}]
We can take a subsequence $\{y_{k_j}\}$ of  $\{y_{i}\}$ such that 
$$ \Omega_{k_{j}} \leq \Omega_{k_{j+1}}
$$
since $\Omega_{i} \to \Delta$ when $i \to \infty$. From Conjecture \ref{gacc}, by taking a subsequence again, we may assume that it holds that
$$\Omega_{k_{0}}^{Y_{k_0}}|_{F_{k_0}} = \Omega_{k_{l}}^{Y_{k_0}}|_{F_{k_0}}$$
for a general fiber $F_{k_0}$ of $f_{k_0}$ and $l>0$ since all coefficients of $\Omega_{k_{j}}^{Y_{k_j}}|_{F_{k_j}}$ have only finitely many possibilities. Let $i=k_0$, then we see that 
$$K_{Y_i}+\Delta^{Y_i} \equiv_{f_i} 0.
$$
\end{proof}
Thus we construct such a model as in Lemma \ref{term pe fibration lemma}. 
\end{proof}

\begin{rem}\label{rem-real-non-non-posi} We do not know whether the above birational map $\varphi$ is $(K_X+\Delta)$-non-positive or not.

\end{rem}

\begin{proof}[Proof of Theorem \ref{main-real}] We will show it by induction on dimension. In particular we may assume that Conjecture \ref{conj-non-vani-k} in the case where the dimension $\leq n-1$ and $\mathbb{K}=\mathbb{R}$ holds. Now we may assume that $(X,\Delta)$ is a $\mathbb{Q}$-factorial divisorial log terminal pair due to a {\em dlt blow-up} (cf. \cite[Theorem 3.1]{kokov-lc-dubois}, \cite[Theorem 10.4]{fujino-fund} and \cite[Section 4]{fujino-ss}). First we show Theorem 
 \ref{main-real} in the following case.
\begin{case}\label{QtoRinKLT} $(X,\Delta)$ is kawamata log terminal and $\Delta$ is an $\mathbb{R}$-divisor.

\end{case}

\begin{proof}[Proof of Case \ref{QtoRinKLT}]\label{proofofQtoRinKLT}We may assume that we can take a sequence of effective $\mathbb{Q}$-divisors $\{\Delta_i\}$ such that $\Delta_i\leq \Delta_{i+1}$, $K_X+\Delta_i$ is not pseudo-effective for any $i\geq 0$, and
$$\lim_{i \to \infty}\Delta_i=\Delta.$$
By Lemma \ref{term pe fibration lemma}, we can take a contracting birational map $\varphi:X \dashrightarrow X'$ such that there exists a projective morphism $f':X' \to Z$ with connected fibers satisfying:
\begin{itemize}
\item[(1)] $(X',\Delta')$ is $\mathbb{Q}$-factorial log canonical and $\rho(X'/Z)=1$,

\item[(2)] $K_{X'}+\Delta' \equiv_{f'} 0$, 

\item[(3)] $\Delta'-\Delta'_i$ is $f'$-ample for some $i$, and

\item[(4)] $\mathrm{dim}\,X > \mathrm{dim}\,Z,$

\end{itemize}
where $\Delta'$ and $\Delta'_i$ are the strict transform of $\Delta$ and $\Delta_i$ on $X'$. By taking resolution of $\varphi$, we may assume that $\varphi$ is morphism. Thus we see that $\kappa_{\sigma}((K_X+\Delta)|_{F})=0$ for a general  
fiber of $f' \circ \varphi$, where $\kappa_{\sigma}(\cdot)$ is the numerical dimension (cf. \cite{nakayama-zariski-abun} and \cite{lehmann-comparing}). When $\dim\,Z=0$, we see that $\kappa_{\sigma}(K_X+\Delta)=0$. Then, from the abundance theorem of numerical Kodaira dimension zero for $\mathbb{R}$-divisors (cf. \cite[Theorem 4.2]{ambro-canonical}, \cite[V, 4.9. Corollary]{nakayama-zariski-abun}, \cite[Corollaire 3.4]{duruel-nu0}, \cite{kawamata-abunnuzero}, \cite{ckp-num}, and \cite[Theorem 1.3]{g3}), we may assume that $\dim\,Z \geq1$. Then, by Remark \ref{rem_real_bir1} and Kawamata's theorem (cf. \cite[Theorem 1.1]{fujino-kawamata}, \cite[Theorem 6-1-11]{kamama}), there exists a good minimal model $f'':(X'',\Delta'')\to Z$ of $(X,\Delta)$ over $Z$. And let $g:X'' \to Z'$ be the morphism of the canonical model $Z'$ of $(X,\Delta)$. Then $Z' \to Z$ is a birational morphism. From Ambro's canonical bundle formula for $\mathbb{R}$-divisors (cf.~\cite[Theorem 4.1]{ambro-canonical} and \cite[Theorem 3.1]{fg2}) there exists an effective divisor $\Gamma_{Z'}$ on $Z'$ such that $K_{X''}+\Delta'' \sim_{\mathbb{R}}g^*(K_{Z'}+\Gamma_{Z'})$. By the hypothesis on induction, we can take an effective divisor $D'$ on $Z'$ such that $K_{Z'}+\Gamma_{Z'}\sim_{\mathbb{R}}D'$. 
\end{proof}
Next we show Theorem \ref{main-real} in the case where $(X,\Delta)$ is divisorial log terminal and $\Delta$ is an $\mathbb{R}$-divisor. 
\begin{case}\label{QtoRinDLT} $(X,\Delta)$ is divisorial log terminal and $\Delta$ is an $\mathbb{R}$-divisor.

\end{case}
\begin{proof}[Proof of Case \ref{QtoRinDLT}]\label{proofofQtoRinDLT}
We take a decreasing sequence $\{\epsilon_i\}$ of positive numbers such that $\lim_{i \to \infty}\epsilon_i=0$. Let $S=\sum S_k$ or $0$ be the reduced part of $\Delta$, $S_k$ its components, and $\Delta_i=\Delta-\epsilon_iS$. We show Theorem \ref{main-real} by induction on the number $r$ of the components of $S$. If $r=0$, Case \ref{QtoRinKLT} implies Conjecture \ref{conj-non-vani-k} for $K_X+\Delta$. When $r>0$, we may assume that $K_{X} +\Delta_i$ is not pseudo-effective from Case \ref{QtoRinKLT} and $K_X+\Delta-\delta S_k$ is not pseudo-effective for any $k$ and $\delta >0$. Then by Lemma \ref{term pe fibration lemma} we can take a contracting birational map $\varphi:X \dashrightarrow X'$ such that there exists a projective morphism $f':X' \to Z$ with connected fibers satisfying:
\begin{itemize}
\item[(1)] $(X',\Delta')$ is $\mathbb{Q}$-factorial log canonical and $\rho(X'/Z)=1$,

\item[(2)] $K_{X'}+\Delta' \equiv_{f'} 0$, 

\item[(3)] $\Delta'-\Delta'_i$ is $f'$-ample for some $i$, and

\item[(4)] $\mathrm{dim}\,X > \mathrm{dim}\,Z,$

\end{itemize}
where $\Delta'$ and $\Delta'_i$ are the strict transform of $\Delta$ and $\Delta_i$ on $X'$. Take log resolutions $p: W \to X$ of $(X,\Delta)$ and $q:W \to X'$ of $(X',\Delta')$ such that $\varphi \circ p=q$. Set the effective divisor $\Gamma$ satisfying
$$K_W+\Gamma =p^*(K_X+\Delta)+E,
$$
where $E$ is an effective divisor such that $E$ has no common components with $\Gamma$. Set the strict transform $\widetilde{S_k}$ and $\widetilde{S}$ of $S_k$ and $S$ respectively on $W$. From Lemma \ref{term pe fibration lemma} (3) $\Supp\, \widetilde{S}$ dominates $Z$.  By the same arguments as the proof of Case \ref{QtoRinKLT}, we may assume that $\dim\,Z\geq 1.$
Then, by Remark \ref{rem_real_bir1}, the abundance conjecture (\ref{conj-abun-nef}) in dimension $\leq n-1$, and \cite[Theorem 4.12]{fg3} (cf. \cite[Corollary 6.7]{fujino-bpf}), there exists a good minimal model $f':(W',\Gamma')\to Z$ of $(W,\Gamma)$ in the sense of Birkar--Shokurov over $Z$. If some $\widetilde{S_k}$ contracts by the birational map $W \dashrightarrow W'$ (may not be contracting), then  $K_W+\Gamma-\delta \widetilde{S_k}$ is pseudo-effective for some $\delta >0$ from the positivity property of the definition of minimal models (cf. Definition\,\ref{real-def-minimalmodelsenseofBS}). Thus $K_X+\Delta-\delta S_k( = p_*(K_W+\Gamma-\delta \widetilde{S_k}))$ is also pseudo-effective. But this is a contradiction to the assumption of $(X,\Delta)$. Thus we see that any $\widetilde{S_k}$ dose not contract by the birational map $W \dashrightarrow W'$. Let $g:W' \to Z'$ be the morphism of the canonical model $Z'$ of $(W,\Gamma)$. Then $Z' \to Z$ is a birational morphism since $\kappa_{\sigma}((K_W+\Gamma)|_{F})=0$ for a general fiber $F$ of $f' \circ q$.
Thus some strict transform $T_k$ of $\widetilde{S_k}$ on $W'$ dominates $Z'$. Now $K_{W'}+\Gamma' \sim_{\mathbb{R}} g^*C$ for some $\mathbb{R}$-Cartier divisor $C$ on $Z'$. By hypothesis of the induction on dimension, there exists an effective $\mathbb{R}$-divisor $D_{T_k}$ on $T_k$ such that 
$$(K_{W'}+\Gamma')|_{T_k}=K_{T_k}+\Gamma_{T_k} \sim_{\mathbb{R}}D_{T_k}.$$
Since $T_k$ dominates $Z'$, there also exists some effective $\mathbb{R}$-divisor $G$  such that $G \sim_{\mathbb{R}}C$. Thus $$K_{W'}+\Gamma' \sim_{\mathbb{R}} g^*G\geq0.$$ This implies the non-vanishing of $K_X+\Delta$. \end{proof}
 We finish the proof of Theorem \ref{main-real}.

\end{proof}

\section{Log non-vanishing theorem for rationally connected varieties}\label{RC-non}

From the same argument as the proof of Case \ref{QtoRinKLT} we see the following theorem:

\begin{thm}\label{main-RC-non-van}Assume that the global ACC conjecture (\ref{gacc}) and the ACC conjecture for log canonical thresholds (\ref{acclct}) in dimension $\leq n$. Let $X$ be a rationally connected variety of dimension $n$ and $\Delta$ an effective $\mathbb{Q}$-Weil divisor such that $K_X+\Delta$ is $\mathbb{Q}$-Cartier and $(X,\Delta)$ is kawamata log terminal. If $K_X+\Delta$ is pseudo-effective,  then there exists an effective $\mathbb{Q}$-Cartier divisor $D$ such that $D \sim_{\mathbb{Q}} K_X+\Delta$.
\end{thm}

\begin{proof}We show by induction on dimension. First, we may assume that $X$ is smooth by taking a log resolution of $(X,\Delta)$. From \cite[Proposition 8.7]{dhp-ext}, the pseudo-effective threshold of $\Delta$ for $K_X$ is also a rational number. Thus we may assume that $K_X+\Delta-\epsilon \Delta$ is not pseudo-effective for any positive number $\epsilon$. We take a decreasing sequence $\{\epsilon_i\}$ of positive numbers such that $\lim_{i \to \infty}\epsilon_i=0$. Let $\Delta_i=\Delta-\epsilon_i \Delta$. Then, by the same argument as the proof of Case \ref{QtoRinKLT}, we may assume that there exists a projective morphism $f:X \to Z$ of connected fibers to normal variety $Z$ such that $\kappa_{\sigma}((K_X+\Delta)|_{F})=0$ for a general fiber $F$ of $f$ and $\mathrm{dim}\,X > \mathrm{dim}\,Z$. Moreover we see that $Z$ is also a rationally connected variety. Then we see that Theorem \ref{main-RC-non-van} by \cite[Lemma\,4.4]{gl1} (cf. \cite{laigood}) and the hypothesis of induction. 
\end{proof}


\begin{thebibliography}{kamama}


\bibitem[A]{ambro-canonical}
F. Ambro, The moduli $b$-divisor of an lc-trivial fibration, 
Compos. Math. {\textbf{141}} (2005), no. 2, 385--403.

\bibitem[B1]{bir-exiI}

C. Birkar, On existence of log minimal models, Compositio Mathematica {\textbf{146}} (2010), 919--928.

\bibitem[B2]{bir-exiII}

C. Birkar, On existence of log minimal models II, J. Reine Angew Math. ,{\textbf{658}} (2011), 99--113.


\bibitem[BCHM]{bchm}
C. Birkar, P. Cascini, C. D. Hacon and J. $\mathrm{M^{c}}$Kernan, Existence of minimal models for varieties of log general type, J. Amer. Math. Soc. {\textbf{23}} (2010), 405--468.

\bibitem[BS]{bSh}
C. Birkar and V. V. Shokurov,  Mld's vs thresholds and flips. J. Reine Angew. Math. {\textbf{638}} (2010), 209--234.

\bibitem[CKP]{ckp-num} 
F.~Campana, V.~Koziarz, and 
M.~P\u{a}un, Numerical character of the effectivity 
of adjoint line bundles, preprint (2010). 

\bibitem[DHP]{dhp-ext}
J-P.~Demailly,~C.~D.~Hacon,~and~M.~P{\u{a}}un,~Extension theorems, Non-vanishing and the existence of good minimal models, 2010, arXiv:1012.0493v2

\bibitem[D]{duruel-nu0}
S. Druel, Quelques remarques sur la d\'ecomposition de Zariski divisorielle sur les vari\'et\'es dont la premi\'ere classe de Chern est nulle, Math. Z., {\textbf{267}}, 1-2 (2011), p. 413--423.
　
\bibitem[F1]{fujino-kawamata} 
O.~Fujino, 
On Kawamata's theorem, {\em{Classification of 
Algebraic Varieties}}, 305--315, 
EMS Ser. of Congr. Rep., Eur. Math. Soc., Z\"urich, 2010.

\bibitem[F2]{fujino-ss} 
O.~Fujino, Semi-stable minimal model program for varieties with trivial 
canonical divisor, Proc. Japan 
Acad. Ser. A Math. Sci. {\textbf{87}} (2011), no. 3, 25--30. 

\bibitem[F3]{fujino-fund}
O.~Fujino, Fundamental theorems for the log minimal model program, 
Publ. Res. Inst. Math. Sci. {\textbf{47}} (2011), no. 3, 727--789.

\bibitem[F4]{fujino-bpf} 
O.~Fujino, Base point free  theorems---saturation, 
b-divisors, and canonical bundle formula---, preprint (2011), to appear in Algebra Number Theory.

\bibitem[FG1]{fg2}
O. Fujino and Y. Gongyo, On canonical bundle formulae and subadjunctions, to appear in Michi. Math. Journal.

\bibitem[FG2]{fg3}
O. Fujino and Y. Gongyo, Log pluricanonical representations and abundance conjecture, preprint (2011)


\bibitem[G]{g3}

Y. Gongyo, On the minimal model theory for dlt pairs of numerical log Kodaira dimension zero, Math. Res. Lett. {\textbf{18}} (2011), no. 5, 991--1000

\bibitem[GL]{gl1}

Y. Gongyo and B. Lehmann, Reduction maps and minimal model theory, preprint (2011).

\bibitem[Ka]{kawamata-abunnuzero}
Y. Kawamata, On the abundance theorem in the case of $\kappa_{\sigma}=0$, preprint, to appear in Amer. J. Math.

\bibitem[KaMM]{kamama}
Y. Kawamata, K, Matsuda and K, Matsuki, \textit{Introduction to the minimal model problem}, Algebraic geometry, Sendai, 1985,  283--360, Adv. Stud. Pure Math., 10, North-Holland, Amsterdam, 1987. 

\bibitem[KoM]{komo}
J. Koll\'ar and S. Mori, \textit{Birational geometry of algebraic varieties}, Cambridge Tracts in Math., {\textbf{134}} (1998).

\bibitem[KoKov]{kokov-lc-dubois}
J.~Koll\'ar and S.~J.~Kov\'acs,  
Log canonical singularities are Du Bois, 
J. Amer. Math. Soc. {\textbf{23}}, no. 3, 791--813.

\bibitem[Lai]{laigood}
C.-J.~Lai, Varieties fibered by good minimal models, Math. Ann., {\textbf{350}} (2011), no. 3, 533--547.

\bibitem[Leh2]{lehmann-comparing}
B. Lehmann, Comparing numerical dimensions, arXiv:1103.0440v1.

\bibitem[N]{nakayama-zariski-abun}
N. Nakayama, \textit{Zariski decomposition and abundance}, MSJ Memoirs, {\textbf{14}}. Mathematical Society of Japan, Tokyo, 2004. 




\end{thebibliography}
\end{document}